\long\def\symbolfootnote[#1]#2{\begingroup\def\thefootnote{\fnsymbol{footnote}}\footnote[#1]{#2}\endgroup}

\documentclass{amsart}
\usepackage{hyperref}
\usepackage{indentfirst}
\usepackage{amssymb}
\usepackage{amsmath}
\usepackage{amsfonts}

\newtheorem{theorem}{Theorem}[section]
\newtheorem{corollary}[theorem]{Corollary}
\newtheorem{lemma}[theorem]{Lemma}
\theoremstyle{remark}
\newtheorem{remark}[theorem]{Remark}
\theoremstyle{definition}

\newtheorem{example}[theorem]{Example}
\theoremstyle{proposition}
\newtheorem{proposition}[theorem]{Proposition}
\numberwithin{equation}{section}

\begin{document}
\author{Xiaohuan Mo and Linfeng Zhou}
\title[The curvatures of spherically symmetric Finsler metrics]{The curvatures of spherically symmetric Finsler metrics in $R^n$ }
\date{}
\maketitle

\begin{abstract}
In this paper, we classify the spherically symmetric Berwald metrics in $\mathbb{R}^n$.  For the spherically symmetric Landsberg metrics, we prove that there do not exist any non-Berwald metrics among the regular case. The partial differential equation systems which can respectively characterize the spherically symmetric Finsler metrics with constant flag curvature and Einstein metrics of this type is also obtained. Utilizing these equations, we find an effective way to construct the non-projective, non-Randers Finsler metrics with constant flag curvature and many explicit examples are given by this method.
\\

\end{abstract}

\section{Introduction}
In the Finsler geometry, there exists a long open problem that the geometers are eager to know: whether there exists a Landsberg metric which does not need to be Berwald type. This problem is called an unicorn problem by D. Bao and M. Matsumoto declared that the search for such metrics represents the next frontier of Finsler geometry \cite{Ba}.  

In the first instance, it is believed that all Landsberg metrics have to be Berwaldian since for Randers metrics, there is no exceptional case and no other counterexamples are found for a quite long time. In 2008,  Z. I. Szab\'{o} claims that all regular Landsberg metrics are Berwald  type by using average metric \cite{Sz}.  However, there is a gap in his proof and one can consult the paper \cite{Sz1} and \cite{Ma} for the details. 

On the other hand, more and more clues indicate there might exist an unicorn metric which means non-Berwaldian Landsberg type metric. In fact, since 2002, R. Bryant has claimed that there exit such generalized metrics in two dimensions, depending on two families of functions of two variables \cite{Ba}.  In succession, G.S. Asanov discovers the $y$-local unicorn metrics from Physics in 2006\cite{As} \cite{As1}.  His examples belong to $(\alpha,\beta)$-metrics and are singular at some $y$-directions, thus it is not $y$-global \cite{Ba}. Since 2004, Z. Shen has been working on the classification of  $(\alpha,\beta)$-metrics with Landsberg type and proves that there is no unicorn among all regular $(\alpha,\beta)$-metrics. Later, motivated by Asanov's examples, Z. Shen characterizes almost regular Landsberg $(\alpha, \beta)$-metrics which generalize Asanov's results \cite{Sh}.

One may naturally question :  how about those non-$(\alpha, \beta)$-metrics? Is it  possible to search an unicorn metric among them? 

In this paper, we try to find the unicorn metrics in the class of  spherically symmetric Finsler metrics in $R^n$.  The spherically symmetric Finsler metrics are introduced by the second author in \cite{Zh} and it is not fully studied till now.  They are a class of Finsler metrics which have  a rotation symmetry and include many Finsler metrics: such as Riemannian space forms, some Randers metrics, Bryant metrics and the example of Berwald. Furthermore, they usually do not need to be $(\alpha,\beta)$-metrics.  Actually, we give the following explicit form of all almost regular Landsberg metrics of this type:

\begin{theorem}\label{lt} Let $(\Omega ,F)$ be an almost regular spherically symmetric Finsler metric in $\mathbb{R}^n$ and $n\geq 3$.  If $F$ is a Landsberg metric, then either
\begin{enumerate}
\item[(1)] $F$ is Berwaldian or \\
 \item[(2)] there exist the smooth functions $c_1, c_2, c_3$ of $r$ such that
 \[F=u\exp(\int \frac{(c_1+2c_3)r^2s+2c_2\sqrt{r^2-s^2}}{r^2+(c_1+2c_3)r^2s^2-2c_3r^4+2c_2s\sqrt{r^2-s^2}}ds).\]
 \end{enumerate}
 Here we denote $r=|x|$, $u=|y|$, $s=\frac{\langle x,y\rangle}{|y|}$.
\end{theorem}

\noindent  Notice that the metrics in the second case are singular at some tangent directions. Thus we prove that all regular Landsberg metrics of this class must be Berwaldian. 

Another fundamental topic in Finsler geometry is to classify or look for  the metrics with constant flag curvature. Till now, as far as our knowledge, a local classification theorem is merely done for Randers metrics \cite{BRS}  and square metrics \cite{Zh0} \cite{SY} and both of them are $(\alpha,\beta)$-metrics. It is still open that if there exist other type of $(\alpha,\beta)$-metrics with constant flag curvature.  If imposing the condition of local projective flatness, B. Li and Z. Shen prove that this kind of $(\alpha,\beta)$-metrics are exactly either the Randers type or the square type \cite{LS}.

Furthermore, except for the Randers metrics, the examples with constant flag curvature we know are all locally projectively flat.  Can we find a method to construct a Finsler metric on a manifold which is of scalar flag curvature, but not Randers, not locally projectively flat? This is an open problem proposed by Z. Shen \cite{Sh1}. 

In this paper, we calculate the Riemann curvature of spherically symmetric Finsler metrics and obtain three characterized partial differential equations of those metrics of constant flag curvature, a partial differential equation to characterize Einstein metrics. From them, we can observe some specific solutions and find an effective way to construct some non-Randers, non-projective Finsler metrics and all of its flag curvatures are constant.  For the convenience, two examples are listed here:
\begin{enumerate}
\item[(1)]The Finsler metric 
\[F:=u\frac{(2r+1)^2}{(4r+1)^{\frac{3}{2}}}e^{\Big(\int_{0}^s\frac{\pm4r(r+4r^2-2s^2)-4s(1+2r)\sqrt{r\big(r+4r^2-4s^2\big)}}{\Big(r+4r^2-4s^2\Big)\Big(\pm2rs+(1+2r)\sqrt{r\big(r+4r^2-4s^2\big)}\Big)}ds\Big)}\]
 defined on $\Omega=\mathbb{R}^n\setminus\{0\}$ has a vanishing flag curvature. Here $r=|x|$, $u=|y|$, $s=\frac{\langle x,y\rangle}{|y|}$.
 \item[(2)]
 The Finsler metric 
\[F^2:=u^2\Big(\frac{1}{4r+1}\pm\frac{4\sqrt{r(r+4r^2-4s^2)}s}{r(2r+1)(4r+1)^2}-\frac{4(4r^2+3r+1)}{r(2r+1)^2(4r+1)^2}s^2\Big)\]
 defined on $\Omega=\mathbb{R}^n\setminus\{0\}$ has a negative flag curvature $K=-1$. Here $r=|x|$, $u=|y|$, $s=\frac{\langle x,y\rangle}{|y|}$.
 \end{enumerate}
 It is reasonable to looking forward to finding more examples of constant flag curvature in this class, especially those metrics with flag curvature $K=1$.

The paper is organized as follows. In section 2, we introduce some preliminary definitions of a spherically symmetric Finsler metric and derive a formula of its the geodesic spray coefficients, which is repeatedly used in latter sections. The Berwald curvature tensor of a spherically symmetric Finsler metric is computed in section 3. Moreover, we completely determine the metric function of a spherical symmetric Berwald metric. 
In section 4, the Landsberg curvature is discussed and our main Theorem \ref{lt} is proved. In the section 5,  from the constant flag curvature equations, we give the equations characterizing the spherically symmetric Finsler metrics with constant flag curvature and Einstein metrics of this type. Via these equations, a procedure of 4 steps on how to construct the example is elaborated  and some explicit examples are obtained by this method in final section 6.

\section{Preliminary} 
Let $F$ be a Finsler metric defined on  a domain $\Omega$ which is contained in $\mathbb{R}^n$. $F$ is called spherically symmetric if it is invariant under any rotations in $\mathbb{R}^n$. According to the equation of Killing fields, we know that there exists a positive function $\phi$ depending on two variables so that $F$ can be written as
$F=|y|\phi(|x|,\frac{\langle x,y\rangle}{|y|})$
where $x$ is a point in the domain $\Omega$, $y$ is a tangent vector at the point $x$ and $\langle , \rangle$, $|\cdot|$ are standard inner product and norm in Euclidean space. One can see the details in \cite{Zh}. For our convenience, denote 
$r=|x|,u=|y|,$ $v=\langle x, y\rangle,s=\frac{\langle x,y \rangle}{|y|}.$
Then $F$ has the expression $F=u\phi(r,s)$ and it does not always need to be $(\alpha, \beta)$-metrics such as the Bryant metrics.

Since the metric tensor $g_{ij}:=\frac{1}{2}\frac{\partial^2 F^2}{\partial y^i\partial y^j}$ does not involve in the derivative on the point $x$, it shares the same formula with the one of the $(\alpha, \beta)$-metrics:
\begin{eqnarray*}g_{ij}&=&\phi(\phi-s\phi_s)\delta_{ij}+(\phi_s^2+\phi\phi_{ss})x^ix^j+[s^2\phi\phi_{ss}-s(\phi-s\phi_s)\phi_s]\frac{y^i}{u}\frac{y^j}{u}\\
&&+[(\phi-s\phi_s)\phi_s-s\phi\phi_{ss}](x^i\frac{y^j}{u}+x^j\frac{y^i}{u}).
\end{eqnarray*}

A spherically symmetric Finsler metric $F=|y|\phi(|x|,\frac{\langle x,y\rangle}{|y|})$ is called an almost regular metric if it has the following properties: (i) $F(x,y)>0$ and (ii) the metric tensor $g_{ij}(x,y)>0$ for any $y\in T_x\Omega$ with $\langle x,y\rangle<|x||y|$. $F$ might be singular for $y=kx$. Such function $F$ is called an almost regular spherically symmetric Finsler metric. 

In order to compute the geodesic spray coefficients of a spherically symmetric Finsler metric $F$, let us denote 
\[g_{ij}=\rho \delta_{ij}+\rho_0x^ix^j+\rho_1(x^i\frac{y^j}{u}+x^j\frac{y^i}{u})+\rho_2\frac{y^i}{u}\frac{y^j}{u}\]
where $$\rho=\phi(\phi-s\phi_s), \rho_0=\phi_s^2+\phi\phi_{ss}, \rho_1=(\phi-s\phi_s)\phi_s-s\phi\phi_{ss}, \rho_2=s^2\phi\phi_{ss}-s(\phi-s\phi_s)\phi_s.$$ Therefore, the inverse of the metric tensor is given by
$$g^{ij}=\rho^{-1}(\delta^{ij}-\tau x^ix^j-\eta Y^iY^j)$$
where $$Y^i=\frac{y^i}{u}+\lambda x^i, \lambda=\frac{\epsilon-\delta s}{1+\delta r^2}, \eta=\frac{\mu}{1+Y^2\mu}, Y=\sqrt{1+(\lambda+\epsilon)s+\lambda\epsilon r^2}$$ and $$\epsilon=\frac{\rho_1}{\rho_2}, \delta=\frac{\rho_0-\epsilon^2\rho_2}{\rho}, \mu=\frac{\rho_2}{\rho}, \tau=\frac{\delta}{1+\delta r^2}.$$ On the other hand, by the definition of the geodesic spray coefficients, we have
\[G^i:=\frac{1}{4}g^{il}\{(F^2)_{x^ky^l}y^k-(F^2)_{x^l}\}=\frac{F_{x^k}y^k}{2F}y^i+\frac{F}{2}g^{il}\{F_{x^ky^l}y^k-F_{x^l}\}.\]
Since $F_{x^k}=u\phi_r\frac{x^k}{r}+\phi_sy^k$, one can write the first part as
\begin{equation}\label{pr1}\frac{F_{x^k}y^k}{2F}y^i=\frac{u}{2\phi}(\frac{s}{r}\phi_r+\phi_s)y^i.\end{equation}
At the same time, it can be computed that
\[F_{x^ky^l}=(u\phi_r\frac{x^k}{r}+\phi_sy^k)_{y^l}=\frac{\phi_r}{ru}x^ky^l+\frac{u}{r}\phi_{rs}(\frac{x^l}{u}-\frac{v}{u^3}y^l)x^k+\phi_{ss}(\frac{x^l}{u}-\frac{v}{u^3}y^l)y^k+\phi_s\delta^k_{\ l}.\]
Hence
\begin{equation}\label{pr2}
F_{x^ky^l}y^k-F_{x^l}=(-ux^l+sy^l)(\frac{\phi_r}{r}-\frac{s}{r}\phi_{rs}-\phi_{ss}).\end{equation}
Combining (\ref{pr1}) and (\ref{pr2}), the geodesic spray coefficients become
\[G^i=\frac{u}{2\phi}(\frac{s}{r}\phi_r+\phi_s)y^i+\frac{u\phi}{2}g^{il}(-ux^l+sy^l)(\frac{\phi_r}{r}-\frac{s}{r}\phi_{rs}-\phi_{ss}).\]
So we only need to compute
\begin{eqnarray*}
g^{il}(-ux^l+sy^l)&=&\rho^{-1}(\delta^{il}-\tau x^ix^l-\eta Y^iY^l)(-ux^l+sy^l)\\
&=&\rho^{-1}\{[-u+\tau u(r^2-s^2)+\lambda^2\eta u(r^2-s^2)]x^i+[s+\lambda\eta (r^2-s^2)]y^i\}.
\end{eqnarray*}
Plugging above equalities into $G^i$ and simplifying it, one will finally come to the formula
\[G^i=uPy^i+u^2Qx^i\]
where 
$$P=-\frac{1}{\phi}\big(s\phi+(r^2-s^2)\phi_s\big)Q+\frac{1}{2r\phi}(s\phi_r+r\phi_s)$$
and 
$$Q=\frac{1}{2r}\frac{-\phi_r+s\phi_{rs}+r\phi_{ss}}{\phi-s\phi_s+(r^2-s^2)\phi_{ss}}.$$
This formula is a little more complicated than the one of $(\alpha,\beta)$-metrics because the partial derivatives of the point $x$ are involved in. 

For a spherically symmetric Finsler metric $F$ in $\mathbb{R}^n$, it is called projective if  its geodesics are straight lines. It is easy to check that $F$ is projective if and only if in its geodesic spray coefficients $Q=0$. In \cite{Zh}, the second author gives the explicit formula of projective spherically symmetric Finsler metrics.  Moreover, in \cite {Zh1}, this type of metrics with constant flag curvature are completely classified. 

Since the metrics have a nice symmetry, the geodesic spray coefficients look clean and neat. Furthermore, here is a philosophy we will often use  in the section 4 and the section 5: once one can determine the specific $P$ and $Q$ in the geodesic spray coefficients, by a technique, one can immediately obtain the original metric $F$ by solving a linear 1-order partial differential equation system.   

\section{Berwald Curvature}
The Berwald curvature of a Finsler metric is a tensor defined in local coordinates as follows:
\[B:=B^i_{\ jkl}dx^j\otimes dx^k\otimes dx^l\otimes \frac{\partial}{\partial x^i}\]
where $B^i_{\ jkl}=\frac{\partial^3G^i}{\partial y^j \partial y^k\partial y^l }$ and $G^i$ is the geodesic spray coefficients. For a spherically symmetric Finsler metric $F=u\phi(r,s)$, we already know its geodesic spray coefficients can be written as $G^i=uPy^i+u^2Qx^i$. Plugging it into the definition of Berwald curvature, one can calculate $B^i_{\ jkl}$ given by
\begin{eqnarray*}
B^i_{\ jkl}&=&\frac{P_{ss}}{u}(\delta^i_{\ j} x^kx^l+\delta^i_{\ l}x^jx^k+\delta^i_{\ k}x^jx^l)+(\frac{P}{u}-\frac{s}{u}P_s)(\delta^i_{\ j}\delta_{kl}+\delta^i_{\ k}\delta_{jl}+\delta^i_{\ l}\delta_{jk})\\
&&-\frac{s}{u^2}P_{ss}\big(\delta^i_{\ j}(x^ky^l+x^ly^k)+\delta^i_{\ k}(x^jy^l+x^ly^j)+\delta^i_{\ l}(x^jy^k+x^ky^j)\big)\\
&&-\frac{s}{u^2}P_{ss}y^i(\delta_{jk}x^l+\delta_{jl}x^k+\delta_{kl}x^j)+(\frac{Q_s}{u}-\frac{s}{u}Q_{ss})x^i(\delta_{jk}x^l+\delta_{jl}x^k+\delta_{kl}x^j)\\
&&+(\frac{s^2}{u^3}P_{ss}+\frac{s}{u^3}P_s-\frac{P}{u^3})(\delta^i_{\ j}y^ky^l+\delta^i_{\ k}y^jy^l+\delta^i_{\ l}y^jy^k)\\
&&+(\frac{s^2}{u^3}P_{ss}+\frac{s}{u^3}P_s-\frac{P}{u^3})y^i(\delta_{jk}y^l+\delta_{jl}y^k+\delta_{kl}y^j)\\
&&+(\frac{3}{u^5}P-\frac{s^3}{u^5}P_{sss}-\frac{6s^2}{u^5}P_{ss}-\frac{3s}{u^5}P_s)y^iy^jy^ky^l\\
&&+(\frac{s^2}{u^4}P_{sss}+\frac{3s}{u^4}P_{ss})y^i(y^jy^kx^l+y^jy^lx^k+y^ky^lx^j)+\frac{P_{sss}}{u^2}y^ix^jx^kx^l\\
&&-(\frac{P_{ss}}{u^3}+\frac{s}{u^3}P_{sss})y^i(y^jx^kx^l+y^kx^jx^l+y^lx^jx^k)\\
&&+(\frac{s^2}{u^3}Q_{sss}+\frac{s}{u^3}Q_{ss}-\frac{Q_s}{u^3})x^i(x^jy^ky^l+x^ky^jy^l+x^ly^jy^k)\\
&&-\frac{s}{u^2}Q_{sss}x^i(x^jx^ly^k+x^jx^ky^l+x^kx^ly^j)+\frac{Q_{sss}}{u}x^ix^jx^kx^l\\
&&+(\frac{s^2}{u^2}Q_{ss}-\frac{s}{u^2}Q_s)x^i(\delta_{kl}y^j+\delta_{jl}y^k+\delta_{jk}y^l)\\
&&+(\frac{3s}{u^4}Q_s-\frac{3s^2}{u^4}Q_{ss}-\frac{s^3}{u^4}Q_{sss})x^iy^jy^ky^l.
\end{eqnarray*}

As we know, a Finsler metric $F$ is called Berwald metric if the Berwald curvature is zero. 
From above formula, a spherically symmetric metric $F=u\phi(r,s)$ is a Berwald metric if and only if $P$ and $Q$ in its geodesic spray coefficients must satisfy 
\begin{equation}\label{beq1}
\left\{ \begin{array}{l}
         sP_s-P=0\\
         P_{ss}=0\\
         sQ_{ss}-Q_s=0\\
         Q_{sss}=0.
          \end{array} \right.\end{equation}
From these equations, one can first solve $P$ and $Q$, then completely determine the metric function $F$. 

\begin{lemma}  \label{So}
The following 1-order non-homogenous  linear partial differential equation   
\[\frac{\partial \phi(r,s)}{\partial r}+\big(\frac{s}{r}-rc_2(r)(r^2-s^2)s\big)\frac{\partial \phi(r,s)}{\partial s}=-(\frac{1}{r}-rc_2(r)s^2)\phi\]
has the general solution given by
\[\phi=\psi(\frac{s^2}{g(r)+s^2\int 2rc_2(r)g(r)dr})e^{-\int (\frac{2}{r}-r^3c_2(r))dr}s\]
where $g(r)=e^{\int (\frac{2}{r}-2r^3c_2(r))    dr}$.
\end{lemma}

\begin{proof}
It's characteristic equation is 
\[\frac{dr}{1}=\frac{ds}{\frac{s}{r}-rc_2(r)(r^2-s^2)s}=\frac{d\phi}{-(\frac{1}{r}-rc_2(r)s^2)\phi}.\]
From $\frac{dr}{1}=\frac{ds}{\frac{s}{r}-rc_2(r)(r^2-s^2)s}$, we know that
\[\frac{d}{dr}(s^2)=(\frac{2}{r}-2c_2r^3)s^2+2c_2rs^4.\]
This is a Bernoulli equation and can be rewritten as
\[\frac{d}{dr}(\frac{1}{s^2})=(-\frac{2}{r}+2c_2r^3)\frac{1}{s^2}-2c_2r.\]
Thus above equation turns into a linear 1-order ODE of $\frac{1}{s^2}$.  One can easily get its solution
\[\frac{1}{s^2}=e^{\int (-\frac{2}{r}+2c_2r^3)dr}(c-\int 2c_2re^{\int (\frac{2}{r}-2c_2r^3)dr}  dr).\]
Therefore, one first integral of the original equation can be chosen to be
 \[\frac{s^2}{e^{\int (\frac{2}{r}-2c_2r^3)dr}+s^2\int 2c_2re^{\int (\frac{2}{r}-2c_2r^3)dr}dr}=\frac{1}{c}.\]
 In order to get another independent first integral, from the characteristic equation, we notice that
 \[\frac{d \ln s}{-\frac{1}{r}+c_2r^3-c_2rs^2}=\frac{dr}{-1}=\frac{d \ln \phi}{\frac{1}{r}-c_2rs^2}.\]
 It implies that
 \[\frac{d \ln s - d \ln \phi}{-\frac{2}{r}+c_2r^3}=\frac{dr}{-1}.\]
 Integrating above equation yields
 \[\ln \frac{s}{\phi}-\int(\frac{2}{r}-c_2r^3)dr=c.\]
 Obviously, it can be selected as another independent first integral. Hence the general solution of the original equation can be expressed by 
   \[\Psi\big(\frac{s^2}{e^{\int (\frac{2}{r}-2c_2r^3)dr}+s^2\int 2c_2re^{\int (\frac{2}{r}-2c_2r^3)dr}dr},\ln \frac{s}{\phi}-\int(\frac{2}{r}-c_2r^3)dr\big)=0.\]
From above equality, one may solve that
\[\phi=\psi(\frac{s^2}{g(r)+s^2\int 2rc_2(r)g(r)dr})e^{-\int ( \frac{2}{r}-r^3c_2(r)) dr}s,\]
where $g(r)=e^{\int (\frac{2}{r}-2r^3c_2(r))    dr}$.
\end{proof}

\begin{theorem} \label{bth}
Suppose a Finsler metric $F=u\phi(r,s)$ is spherically symmetric in $\mathbb{R}^n$, then $F$ is Berwald metric if and only if either $F$ is Riemannian or there exists a smooth function $c_2(r)$ so that 
\[F=u\psi(\frac{s^2}{g(r)+s^2\int 2rc_2(r)g(r)dr})e^{-\int (\frac{2}{r}-r^3c_2(r))dr}s\]
where $g(r)=e^{\int (\frac{2}{r}-2r^3c_2(r))    dr}$.
\end{theorem}

\begin{proof} Firstly, let us  prove the necessity. Since $F$ is Berwald metric, it satisfies the equations (\ref{beq1}). It means that
there exist three functions $c_1(r)$, $c_2(r)$ and $c_3(r)$ such that
\[P=c_1(r)s, \quad Q=\frac{1}{2}c_2(r)s^2+c_3(r).\]
Note that 
\[P=-\frac{1}{\phi}\big(s\phi+(r^2-s^2)\phi_s\big)Q+\frac{1}{2r\phi}(s\phi_r+r\phi_s),\quad Q=\frac{1}{2r}\frac{-\phi_r+s\phi_{rs}+r\phi_{ss}}{\phi-s\phi_s+(r^2-s^2)\phi_{ss}}.\] 
Plugging $P$ and $Q$ into above equation, one has
\begin{equation*}
\left\{ \begin{array}{l}
-\frac{1}{\phi}\big(s\phi+(r^2-s^2)\phi_s\big)(\frac{1}{2}c_2(r)s^2+c_3(r))+\frac{1}{2r\phi}(s\phi_r+r\phi_s)=c_1(r)s, \\
 \frac{1}{2r}\frac{-\phi_r+s\phi_{rs}+r\phi_{ss}}{\phi-s\phi_s+(r^2-s^2)\phi_{ss}}=\frac{1}{2}c_2(r)s^2+c_3(r).   
 \end{array} \right.
          \end{equation*}
This equation system can be simplified as
\begin{equation}\label{beq2}
\left\{ \begin{array}{l}
\big((r^2-s^2)(2c_3+c_2s^2)-1\big)r\phi_s-s\phi_r+rs(2c_3+c_2s^2)\phi+2rsc_1\phi=0,\\
\big((r^2-s^2)(2c_3+c_2s^2)-1\big)r\phi_{ss}-s\phi_{rs}+\phi_r+r(2c_3+c_2s^2)(\phi-s\phi_s)=0. 
 \end{array} \right.
          \end{equation}
Differentiating the first equation of (\ref{beq2}) with respect to the variable $s$ will conclude that
\[\big((r^2-s^2)(2c_3+c_2s^2)-1\big)r\phi_{ss}-s\phi_{rs}-\phi_r+(3c_2s^2+2c_1+2c_3)r\phi\]
\[+\big(2(r^2-s^2)sc_2-(2c_3+c_2s^2)s+2c_1s\big)r\phi_s=0.\]
From above equation and the second equation of (\ref{beq2}), we can see that
\begin{equation}\label{beq3}
\big((r^2-s^2)sc_2+c_1s\big)r\phi_s-\phi_r+(c_2s^2+c_1)r\phi=0.
\end{equation}
Hence (\ref{beq2}) and (\ref{beq3}) imply that
\[\big((c_1+2c_3)s^2+1-2c_3r^2\big)\phi_s-(c_1+2c_3)s\phi=0.\]

If $(c_1+2c_3)s^2+1-2c_3r^2\neq 0$, integrating above equation concludes 
\begin{eqnarray*}
F=u\phi(r,s)&=&uc\sqrt{(c_1+2c_3)s^2+1-2c_3r^2}\\
&=&c(r)\sqrt{(c_1(r)+2c_3(r))\langle x, y\rangle^2+(1-2c_3(r)r^2)|y|^2}.
\end{eqnarray*}
Therefore, it must be Riemannian.

If $(c_1+2c_3)s^2+1-2c_3r^2=0$ which means $c_3=\frac{1}{2r^2}$ and $c_1=-\frac{1}{r^2}$,  above equation automatically holds and can tell us that the first equation of (\ref{beq2}) implies the second equation of (\ref{beq2}). In this case,  the first equation of  (\ref{beq2}) becomes  
\[ \big(\frac{s}{r}-rc_2(r)(r^2-s^2)s\big)\phi_s+\phi_r=-(\frac{1}{r}-rc_2(r)s^2)\phi.\]
By Lemma \ref{So},  the general solution of the above equation is 
\[\phi=\psi(\frac{s^2}{g(r)+s^2\int 2rc_2(r)g(r)dr})e^{-\int (\frac{2}{r}-r^3c_2(r))dr}s\]
where $g(r)=e^{\int (\frac{2}{r}-2r^3c_2(r))    dr}$.

Now let us prove the sufficiency. If $F$ is Riemmanian, it is of course Berwaldian.  If \[F=u\psi(\frac{s^2}{g(r)+s^2\int 2rc_2(r)g(r)dr})e^{-\int (\frac{2}{r}-r^3c_2(r))dr}s\]
where $c_2(r)$ is a smooth function and $g(r)=e^{\int (\frac{2}{r}-2r^3c_2(r))    dr}$, from above  calculation, we know that its spray coefficients are given by 
\[P=-\frac{s}{r^2}, \quad Q=\frac{1}{2}c_2(r)s^2+\frac{1}{2r^2}.\]
By the equation $(\ref{beq1})$, $F$ is Berwaldian. 
\end{proof}

\section{Landsberg Curvature}

The Landsberg curvature is one part of curvature forms when using Chern connection. In local coordinates, it can be defined as 
$L:=L_{\ jkl}dx^j\otimes dx^k\otimes dx^l$ where
\[L_{jkl}:=-\frac{1}{2}FF_{y^i}\frac{\partial ^3G^i}{\partial y^j\partial y^k\partial y^l}.\]
When a Finsler metric $F=u\phi(r,s)$ is spherically symmetric in $R^n$, substituting the geodesic spray coefficients into above definition and after a not so long calculation, one will obtain
\begin{eqnarray*}L_{jkl}&=&-\frac{\phi}{2}[L_1x^jx^kx^l+L_2(x^j\delta_{kl}+x^k\delta_{jl}+x^l\delta_{jk})+L_3\frac{y^j}{u}\frac{y^k}{u}\frac{y^l}{u}\\
&&+L_4(\frac{y^j}{u}\delta_{kl}+\frac{y^k}{u}\delta_{jl}+\frac{y^l}{u}\delta_{jk})+L_5(\frac{y^j}{u}x^kx^l+\frac{y^k}{u}x^jx^l+\frac{y^l}{u}x^jx^k)\\
&&+L_6(x^j\frac{y^k}{u}\frac{y^l}{u}+x^k\frac{y^j}{u}\frac{y^l}{u}+x^l\frac{y^j}{u}\frac{y^k}{u})]
\end{eqnarray*}
where 
\begin{eqnarray*}
L_1&=&3\phi_sP_{ss}+\phi P_{sss}+\big(s\phi+(r^2-s^2)\phi_s\big)Q_{sss},\\
L_2&=&-s\phi P_{ss}+\phi_s(P-sP_s)+(s\phi+(r^2-s^2)\phi_s)(Q_s-sQ_{ss}),\\
L_3&=&-s^3L_1+3sL_2,\\
L_4&=&-sL_2,\\
L_5&=&-sL_1,\\
L_6&=&s^2L_1-L_2.
\end{eqnarray*}

A Finsler metric is defined as a Landsberg type if its Landsberg curvature vanishes. Obviously, above computation result indicates that $F$ is a Landsberg metric if and only if it satisfies the following equations
\begin{equation}\label{leq1}
\left\{ \begin{array}{l}
          L_1=3\phi_sP_{ss}+\phi P_{sss}+\big(s\phi+(r^2-s^2)\phi_s\big)Q_{sss}=0\\
          \\
         L_2=-s\phi P_{ss}+\phi_s(P-sP_s)+(s\phi+(r^2-s^2)\phi_s)(Q_s-sQ_{ss})=0.
          \end{array} \right.\end{equation}
At first glance, this equation system might not have an explicit solution. However, after a careful analysis and substitution, it can be completely solved. 

\begin{lemma} \label{le1}
 If a spherically symmetric Finsler metric $F=u\phi(r,s)$ in $R^n\ (n\geq 3)$ is Landsberg type, where $u:=|y|$, $r:=|x|$ and $s:=\frac{\langle x,y\rangle}{|y|}$, then there exist the functions $c_0(r), c_1(r), c_2(r), c_3(r)$ so that its geodesic spray coefficients $G^i$ satisfy
\[G^i=uPy^i+u^2Qx^i\]
where 
\[P=c_1(r)s+c_2(r)\frac{\sqrt{r^2-s^2}}{r^2}\]
and 
\[Q=\frac{1}{2}c_0(r)s^2-\frac{c_2(r)s\sqrt{r^2-s^2}}{r^4}+c_3(r).\]
\end{lemma}
\begin{proof} Let $\Theta=P-sP_s$ and $\eta=s\phi+(r^2-s^2)\phi_s$, then 
\[\Theta_s=-sP_{ss},\qquad \eta_s=\phi-s\phi_s+(r^2-s^2)\phi_{ss}.\]
Therefore, the second equation of (\ref{leq1}) implies
\begin{equation}\label{leq2}
(\phi\Theta)_s+\eta(Q_s-sQ_{ss})=0.
\end{equation}
Taking the  derivative with respect to the variable $s$ will obtain
\[(\phi\Theta)_{ss}=s\eta Q_{sss}-\eta_s(Q_s-sQ_{ss}).\]
At the same time, we have
\[(\phi\Theta)_{ss}=-\phi(P_{ss}+sP_{sss})-2s\phi_sP_{ss}+\phi_{ss}\Theta.\]
Plugging the first equation of (\ref{leq1}) into above two equations yields
\begin{eqnarray*}
\eta_s(Q_s-sQ_{ss})&=&(\phi-s\phi_s)P_{ss}-\phi_{ss}\Theta\\
&=&-\frac{\phi-s\phi_s}{s}\Theta_s-\phi_{ss}\Theta.
\end{eqnarray*}
By the equation (\ref{leq2}), one has
\[ \eta(\phi_{ss}\Theta+\frac{\phi-s\phi_s}{s}\Theta_s)-\eta_s(\phi_s\Theta+\phi\Theta_s)=0.\]
Hence,
\begin{equation}\label{leq3}
\frac{\Theta_s}{\Theta}[(\phi-s\phi_s)\eta-s\phi\eta_s]=s\phi_s\eta_s-s\phi_{ss}\eta.
\end{equation}
Clearly,
\begin{eqnarray*}
(\phi-s\phi_s)\eta-s\phi\eta_s&=&(\phi-s\phi_s)\big(s\phi+(r^2-s^2)\phi_s\big)-s\phi\big(\phi-s\phi_s+(r^2-s^2)\phi_{ss}\big)\\
&=&(r^2-s^2)\big((\phi-s\phi_s)\phi_s-s\phi\phi_{ss}\big)
\end{eqnarray*}
and
\begin{eqnarray*}
s\phi_s\eta_s-s\phi_{ss}\eta&=&s\phi_s\big(\phi-s\phi_s+(r^2-s^2)\phi_{ss}\big)-s\phi_{ss}\big(s\phi+(r^2-s^2)\phi_s\big)\\
&=&s(\phi-s\phi_s)\phi_s-s^2\phi\phi_{ss}.
\end{eqnarray*}
Now it is easy to see that the equation (\ref{leq3}) holds if and only if 
\begin{equation}\label{leq4}
(\phi-s\phi_s)\phi_s-s\phi\phi_{ss}=0
\end{equation}
or 
\begin{equation}\label{leq5}
\frac{\Theta_s}{\Theta}(r^2-s^2)=s.
\end{equation}

In the case of the equation (\ref{leq4}), obviously its solution is given by
\[\phi=\sqrt{c_1(r)s^2+2c_2(r)}\]
where $c_1$ and $c_2$ are two functions of the variable $r$. 
Thus the Finsler metric $F=u\phi$ is Riemannian.  

In the case of the equation (\ref{leq5}),  one can obtain 
\[\Theta=P-sP_s=\frac{c_2(r)}{\sqrt{r^2-s^2}}\]
where $c_2(r)$ is a function. This implies
\[P=c_1(r)s+c_2(r)\frac{\sqrt{r^2-s^2}}{r^2},\]
here $c_1(r)$ and $c_2(r)$ are two functions. Combining with the second equation of (\ref{leq1}) concludes
\begin{eqnarray*}
(Q_s-sQ_{ss})\big(s\phi+(r^2-s^2)\phi_s\big)&=&-\frac{c_2(r)s}{(\sqrt{r^2-s^2})^3}\phi-\frac{c_2(r)}{\sqrt{r^2-s^2}}\phi_s\\
&=&-\frac{c_2(r)}{(\sqrt{r^2-s^2})^3}\big(s\phi+(r^2-s^2)\phi_s\big).
\end{eqnarray*}
Since $s\phi+(r^2-s^2)\phi_2\neq0$, $Q$ has to satisfy
\[Q_s-sQ_{ss}=\frac{-c_2(r)}{(\sqrt{r^2-s^2})^3}.\]
Solving this equation will get
\[Q=\frac{1}{2}c_0(r)s^2-\frac{c_2(r)s\sqrt{r^2-s^2}}{r^4}+c_3(r).\]
This completes the proof of the lemma.
\end{proof}

Now we can give a proof of Theorem \ref{lt} by using above lemma. 
\begin{proof} [Proof of Theorem \ref{lt}]
Let us introduce $U$ and $W$ so that
\[U:=\frac{s\phi+(r^2-s^2)\phi_s}{\phi},\qquad W:=\frac{s\phi_r+r\phi_s}{\phi}.\]
First, from the definition of $U$ and $W$, one can solve $\phi_s$ and $\phi_r$:
\begin{equation}\label{leq8}
\phi_s=\frac{U-s}{r^2-s^2}\phi,\qquad\phi_r=\frac{1}{s}(W-\frac{r(U-s)}{r^2-s^2})\phi.
\end{equation}
Plugging $\phi_s$ and $\phi_s$ into $P$ and $Q$, which appear in the geodesic spray coefficients of $F$, we have
\begin{equation}\label{leq6}
\left\{ \begin{array}{l}
 P=-QU+\frac{W}{2r}\\
 \\
Q=\frac{1}{2rs}\frac{2rU-2rs-2r^2W+s(r^2-s^2)W_s+s^2W+sUW}{U^2-sU+(r^2-s^2)U_s}.   
 \end{array} \right.
          \end{equation}
At the same time, since the metric is Landsberg type, according to lemma \ref{le1}, there exist the functions $c_i(r)$ $(i=0,\dots,3)$ such that 
\[ P=c_1(r)s+c_2(r)\frac{\sqrt{r^2-s^2}}{r^2},\quad Q=\frac{1}{2}c_0(r)s^2-\frac{c_2(r)s\sqrt{r^2-s^2}}{r^4}+c_3(r).\]

If $c_2(r)=0$, then $P=c_1(r)s$ and $Q=\frac{1}{2}c_0(r)s^2+c_3(r)$ which means $F$ is Berwaldian. 

If $c_2(r)\neq0$,  substituting $P$ and $Q$ into the (\ref{leq6}) can obtain  
\begin{equation}\label{leq7}
\left\{ \begin{array}{l}
U=\frac{r^2(s+c_1r^2s+2c_2\sqrt{r^2-s^2})}{r^2+2c_2s\sqrt{r^2-s^2}+c_1r^2s^2-2c_3r^2(r^2-s^2)}\\
\\
W=2r(P+UQ).
\end{array}\right.
\end{equation}
From  (\ref{leq8}), one can see that $\phi$ should satisfy
\begin{equation}\label{leq9}
\left\{ \begin{array}{ll}
(\ln\phi)_s=&\frac{(c_1+2c_3)r^2s+2c_2\sqrt{r^2-s^2}}{r^2+(c_1+2c_3)r^2s^2-2c_3r^4+2c_2s\sqrt{r^2-s^2}}\\
\\
(\ln\phi)_r=&\frac{s\sqrt{r^2-s^2}(2c_0c_2r^4+4(c_1+c_3)c_2r^2-2c_2)}{r\big(r^2+(c_1+2c_3)r^2s^2-2c_3r^4+2c_2s\sqrt{r^2-s^2}\big)}\\
&+\frac{c_0c_1r^6s^2+(c_0+4c_1c_3+2c_1^2)r^4s^2-2c_1c_3r^6+c_1r^4}{r\big(r^2+(c_1+2c_3)r^2s^2-2c_3r^4+2c_2s\sqrt{r^2-s^2}\big)}.
\end{array}\right.
\end{equation}
Integrating the first equation of (\ref{leq9}) yields 
\[\phi=\exp(\int \frac{(c_1+2c_3)r^2s+2c_2\sqrt{r^2-s^2}}{r^2+(c_1+2c_3)r^2s^2-2c_3r^4+2c_2s\sqrt{r^2-s^2}}ds).\]

\end{proof}

\begin{remark} The metrics  in the second case of Theorem \ref{lt} are not regular, since along the direction of the radius, we have $s=r$ and the metrics are not $C^2$. \\
\end{remark}

If we assume the metrics to be regular, Theorem \ref{lt} leads to the following corollary.
\begin{corollary}Let $(\Omega ,F)$ be a spherically symmetric Finsler metric in $\mathbb{R}^n$ and $n\geq 3$.  If $F$ is a Landsberg metric, then it must be Berwaldian. 
\end{corollary}

\section{Riemann Curvature}
The Riemann curvature is one of the most important quantities in Finsler geometry and it is defined by
\[R^{i}_{\
j}:=2(G^{i})_{x^{j}}-y^{k}(G^{i})_{x^{k}y^{j}}+2G^{k}(G^{i})_{y^{k}y^{j}}-(G^{i})_{y^{k}}(G^{k})_{y^{j}}.\]
For any tangent plane $P$=span$\{y,u\}\subset T_{x}M$, the flag curvature $K(P,y)$ is defined by
$$K(P,y):=\frac{g_{ij}R^i_{\ k}(x,y)u^ju^k}{F(x,y)^2g_{ij}(x,y)u^iu^j-[g_{ij}(x,y)y^iu^j]^2}$$ which is a generalization of the sectional curvature in Riemannian case. It is well-known that a Finsler metric $F$ has constant flag curvature $K$ if and only if
\begin{equation}\label{feq1}R^i_{\ j}=KF^2(\delta^i_{\ j}-\frac{y^i}{F}F_{y^j}).\end{equation}
The Ricci curvature is the trace of  the Riemann curvature
$\bold{Ric}:=R^{i}_{\ i}.$ A Finsler metric $(M^m, F)$ is called Einstein metric, if there exist a function $K(x)$ so that $$\bold{Ric}=(m-1)K(x)F^2.$$

For a spherically symmetric Finsler metric $F=u\phi(r,s)$, if we note its geodesic spray coefficients as following: $G^i=G_1^i+G_2^i$
where $G_1^i=uPy^i$ and $G_2^i=u^2Qx^i$.
Then its Riemann curvature can be formulated by
\begin{eqnarray*}\
R^i_{\ j}&=&2(G_1^i)_{x^j}+2(G_2^i)_{x^j}-y^k(G_1^i)_{x^ky^j}-y^k(G_2^i)_{x^ky^j}\\
&&+2(G_1^k+G_2^k)\big((G_1^i)_{y^ky^j}+(G_2^i)_{y^ky^j}\big)-\big((G_1^i)_{y^k}+(G_2^i)_{y^k}\big)\big((G^k)_{y^j}+(G^k)_{y^j}\big).
\end{eqnarray*}
Therefore, we need to compute
\begin{eqnarray*}
(G_1^i)_{x^j}&=&\frac{u}{r}P_rx^jy^i+P_sy^iy^j,\\
(G_2^i)_{x^j}&=&u^2Q\delta^i_{\ j}+\frac{u^2}{r}Q_rx^ix^j+uQ_sx^iy^j
\end{eqnarray*}
and 
\begin{eqnarray*}
(G_1^i)_{y^k}&=&uP\delta^i_{\ k}+(\frac{P}{u}-\frac{s}{u}P_s)y^iy^k+P_sx^ky^i,\\
(G_2^i)_{y^k}&=&uQ_sx^ix^k+(2Q-sQ_s)x^iy^k.\end{eqnarray*}
Plugging above equalities into the formula of Rienmann curvature,  one can see that
$$R^i_{\ j}=u^2(R_1\delta^i_{\ j}+R_2\frac{y^i}{u}\frac{y^j}{u}+R_3x^ix^j+R_4x^i\frac{y^j}{u}+R_5x^j\frac{y^i}{u})$$
where
\begin{eqnarray*}
R_1&=&2Q-\frac{s}{r}P_r-P_s+2(r^2-s^2)P_sQ+P^2+2sPQ,\\
R_2&=&P_s-\frac{s}{r}P_r+\frac{s^2}{r}P_{rs}+sP_{ss}-2Q+sQ_s-2sPP_s-4sPQ+4s^2P_sQ-P^2\\
&&-2s(r^2-s^2)P_{ss}Q+3sPP_s+s^2PQ_{s}+(r^2-s^2)sP_sQ_s-2r^2P_sQ,\\
R_3&=&\frac{2}{r}Q_r-Q_{ss}-\frac{s}{r}Q_{rs}+2(r^2-s^2)QQ_{ss}+4Q^2-(r^2-s^2)Q_s^2-2sQQ_s,\\
R_4&=&-\frac{2s}{r}Q_r+\frac{s^2}{r}Q_{rs}+sQ_{ss}-2(r^2-s^2)sQQ_{ss}+(r^2-s^2)sQ_s^2-4sQ^2\\
&&+2s^2QQ_s,\\
R_5&=&\frac{2}{r}P_r-\frac{s}{r}P_{rs}-P_{ss}-Q_s+2PQ-2sP_sQ+2(r^2-s^2)P_{ss}Q-PP_s\\
&&-sPQ_s-(r^2-s^2)P_sQ_s.
\end{eqnarray*}

Notice the constant flag curvature equation (\ref{feq1}), we can yield the following three equations characterizing the spherically symmetric Finsler metrics of constant flag curvature.

\begin{proposition}\label{fprop1} Let $(\Omega ,F)$ be a spherically symmetric Finsler metric in $\mathbb{R}^n$. Then $(\Omega ,F)$ has constant flag curvature $K$ if and only if the following three equations hold
\begin{equation}\label{feq2}
\left\{ \begin{array}{l}
2Q-\frac{s}{r}P_r-P_s+2(r^2-s^2)P_sQ+P^2+2sPQ=K\phi^2\\
\\
\frac{1}{2r}P_r-\frac{s}{2r}P_{rs}-\frac{1}{2}P_{ss}+PQ-sP_sQ+(r^2-s^2)P_{ss}Q=0\\
\\
  \frac{2}{r}Q_r-Q_{ss}-\frac{s}{r}Q_{rs}+2(r^2-s^2)QQ_{ss}+4Q^2-(r^2-s^2)Q_s^2-2sQQ_s=0.
 \end{array} \right.
          \end{equation}
\end{proposition}

\begin{proof}As we know, $(\Omega,F)$ has constant flag curvature if and only if (\ref{feq1}) holds. Since $F$ is spherically symmetric, $F$ can be written as $F=u\phi(r,s)$. Therefore, 
\[F_{y^j}=(\phi-s\phi_s)\frac{y^j}{u}+\phi_sx^j.\]
Plugging $R^i_{\ j}$ and $F_{y^j}$ into (\ref{feq1}), one obtains 
\[R_1\delta^i_{\ j}+R_2\frac{y^i}{u}\frac{y^j}{u}+R_3x^ix^j+R_4x^i\frac{y^j}{u}+R_5x^j\frac{y^i}{u}=K\phi^2\delta^i_{\ j}-K\phi(\phi-s\phi_s)\frac{y^i}{u}\frac{y^j}{u}-K\phi\phi_sx^j\frac{y^i}{u}.\]
This is equivalent to say 
\begin{eqnarray*}
R_1=K\phi^2,\quad R_2=-K\phi(\phi-s\phi_s),\quad R_3=R_4=0,\quad R_5=-K\phi\phi_s.
\end{eqnarray*}
Notice that 
\[R_4=-sR_3\]
and
\[R_2=-R_1-sR_5.\]
Hence the constant flag curvature equation can be reduced as
\[R_1=K\phi^2,\quad R_5=-K\phi\phi_s,\quad R_3=0.\]
That means $F$ has constant flag curvature $K$ if and only if $P$ and $Q$ should satisfy
\begin{equation*}
\left\{ \begin{array}{l}
2Q-\frac{s}{r}P_r-P_s+2(r^2-s^2)P_sQ+P^2+2sPQ=K\phi^2\\
\\
\frac{2}{r}P_r-\frac{s}{r}P_{rs}-P_{ss}-Q_s+2PQ-2sP_sQ\\
+2(r^2-s^2)P_{ss}Q-PP_s-sPQ_s-(r^2-s^2)P_sQ_s=-K\phi\phi_s\\
\\
 \frac{2}{r}Q_r-Q_{ss}-\frac{s}{r}Q_{rs}+2(r^2-s^2)QQ_{ss}+4Q^2-(r^2-s^2)Q_s^2-2sQQ_s=0.
 \end{array} \right.
          \end{equation*}
\\

Furthermore, it is easy to observe that if taking one half of the derivative of the first one of above equations with respect to $s$, then adding the second equation, one can obtain the second equation of (\ref{feq2}). Thus above equations are equivalent to (\ref{feq2}) and the proposition is proved.
\end{proof}

Analogously, we have the following equation to characterize a spherically symmetric Einstein metric.
\begin{proposition}  Let $(\Omega ,F)$ be a spherically symmetric Finsler metric in $\mathbb{R}^n$. $(\Omega ,F)$ is an Einstein metric  if and only if there is a function $K=K(r)$ such that
\begin{equation}\label{feq3}
(n-1)R_1+(r^2-s^2)R_3=(n-1)K\phi^2
\end{equation}
where 
\[R_1=2Q-\frac{s}{r}P_r-P_s+2(r^2-s^2)P_sQ+P^2+2sPQ\]
and \[R_3=\frac{2}{r}Q_r-Q_{ss}-\frac{s}{r}Q_{rs}+2(r^2-s^2)QQ_{ss}+4Q^2-(r^2-s^2)Q_s^2-2sQQ_s.\]
\end{proposition}

\begin{proof} A Finsler metric $F$ is Einstein metric if and only if there exists a function $K(x)$ so that
\[\bold{Ric}=(n-1)K(x)F^2.\]
Since $F$ is spherically symmetric, thus $K(x)$ must be the function of the radius $r$.  Furthermore, the Riemann curvature is given by
\[R^i_{\ j}=u^2(R_1\delta^i_{\ j}+R_2\frac{y^i}{u}\frac{y^j}{u}+R_3x^ix^j+R_4x^i\frac{y^j}{u}+R_5x^j\frac{y^i}{u}).\]
Tracing above equality will obtain the result.
\end{proof}
\begin{remark} From above equation, it is possible to find a clue to verify that Schur lemma still holds for this type of Einstein metrics.
\end{remark}
\section{The explicit examples of constant flag curvature}

Now we are ready to describe the procedure on how to construct a non-projective Finsler metric with constant flag curvature. As for the projective spherically symmetric Finsler metrics, if they have constant flag curvature, the second author has given a complete classification in \cite{Zh1}.

Let $(\Omega, F) \subseteq \mathbb{R}^n$ be a spherically symmetric Finsler metric with constant flag curvature. Assume $F=u\phi(r, s)$ and one part of the geodesic spray coefficients $Q\neq 0$. Therefore, it is not projective.  By proposition \ref{fprop1}, $F$ has constant flag curvature if and only if its geodesic  spray coefficients must satisfy the following three equations
\begin{equation}\label{ege1}
\left\{ \begin{array}{l}
2Q-\frac{s}{r}P_r-P_s+2(r^2-s^2)P_sQ+P^2+2sPQ=K\phi^2\\
\\
\frac{1}{2r}P_r-\frac{s}{2r}P_{rs}-\frac{1}{2}P_{ss}+PQ-sP_sQ+(r^2-s^2)P_{ss}Q=0\\
\\
  \frac{2}{r}Q_r-Q_{ss}-\frac{s}{r}Q_{rs}+2(r^2-s^2)QQ_{ss}+4Q^2-(r^2-s^2)Q_s^2-2sQQ_s=0.
 \end{array} \right.
          \end{equation}

Since it is not easy to get all generic solutions of the third equation explicitly, one possible way is to find some specific solutions. In fact, we can follow a routine step to search the required metrics. 
\flushleft
\begin{enumerate}
\item[\textbf{Step 1}]\textbf{Find some specific solutions of the third equation of (\ref{ege1}) to determine} $\mathbf{Q.}$\\
For example, if suppose $Q=c_1(r)+c_2(r)s^2$ are the solutions of the third equation in (\ref{ege1}),  plugging it into the equation implies $$Q=c_1(r)+\frac{c_1'+2rc_1^2}{r-2r^3c_1}s^2.$$ 
Thus $Q=c_1(r)+\frac{c_1'+2rc_1^2}{r-2r^3c_1}s^2$ is the specific solutions only depending on a function of variable $r$. Furthermore, it usually does not vanish.
\item[\textbf{Step 2}]\textbf{Solve the second equation of (\ref{ege1}) to determine} $\mathbf{P.}$\\
Once $Q$ is precise, the second equation of (\ref{ege1}) can be rewritten as
\[\frac{1}{2r}(P-sP_s)_r+(P-sP_s)Q-\frac{r^2-s^2}{s}(P-sP_s)_sQ+\frac{1}{2s}(P-sP_s)_s=0.\]
In fact, this is a linear 1-order partial differential equation about the function of $P-sP_s$. When $Q$ is not so complicated, it is usually solvable. Thus one can determine  the specific $P$ from it. 
\item[\textbf{Step 3}] \textbf{Compute the flag curvature.} \\
Plugging $P$ and $Q$ in the former steps into the first equation of (\ref{ege1}), if the left hand side equals 0, it means that the metric has a vanishing flag curvature. Otherwise, one can solve $\phi$ and obtain the metric function $F=u\phi$.  By a scaling it has a positive flag curvature 1 if the left hand side of the first equation of (\ref{ege1}) is positive, and it has a negative flag curvature -1 if the left hand side of the first equation of (\ref{ege1}) is negative.
\item[\textbf{Step 4}] \textbf{Check the metric does exist.}\\
When the flag curvature of $F$ does not vanish, one can solve $F$ in step 3. However, we still need to check that the geodesic spray coefficients of $F$ equal $G^i=uPy^i+u^2Qx^i$. Since $P$ and $Q$ are determined, by substituting $\phi$ into the formula of $P$ and $Q$ and comparing them, one will draw a conclusion. 

When the flag curvature of $F$ vanishes, then by a technique used in the proof of Theorem \ref{lt}, we can solve $\phi$ to obtain the metric function from $P$ and $Q$. 
\end{enumerate}

\indent Let us demonstrate the detailed procedure to produce some new examples via above steps.
\begin{example}\label{eg1}
\emph{The Finsler metric 
\[F:=u\frac{(2r+1)^2}{(4r+1)^{\frac{3}{2}}}\exp\Big(\int_{0}^s\frac{\pm4r(r+4r^2-2s^2)-4s(1+2r)\sqrt{r\big(r+4r^2-4s^2\big)}}{\Big(r+4r^2-4s^2\Big)\Big(\pm2rs+(1+2r)\sqrt{r\big(r+4r^2-4s^2\big)}\Big)}ds\Big)\]
 defined on $\Omega=\mathbb{R}^n\setminus\{0\}$ has a vanishing flag curvature. Here $r:=|x|$, $u:=|y|$, $s:=\frac{\langle x,y\rangle}{|y|}$.}
\end{example}
\begin{proof}
Suppose $Q=c_1(r)+c_2(r)s^2$ are the solutions of the third equation in (\ref{ege1}) and plugging it into the equation implies $$Q=c_1(r)+\frac{c_1'+2rc_1^2}{r-2r^3c_1}s^2.$$
If let $c_1(r):=-\frac{1}{r}$, $Q$ can be simplified as $$Q=-\frac{r^2-s^2}{r^3}.$$

Rewrite the second equation of (\ref{ege1}) as 
\[\frac{1}{2r}(P-sP_s)_r+(P-sP_s)Q-\frac{r^2-s^2}{s}(P-sP_s)_sQ+\frac{1}{2s}(P-sP_s)_s=0.\]
When $Q=-\frac{r^2-s^2}{r^3}$, one can solve it and conclude that there must exist a function $f$ such that
\[P-sP_s=f\Big(\frac{r(r^2-s^2)}{r+4r^2-4s^2}\Big)\sqrt{\frac{r}{r+4r^2-4s^2}}.\]
The most simple case is when $f$ equals a constant $c$. This means $$P-sP_s=c\sqrt{\frac{r}{r+4r^2-4s^2}}.$$
Hence there exists a function $g(r)$ so that
$$P=g(r)s+\frac{c\sqrt{r(r+4r^2-4s^2)}}{r(1+4r)}.$$

Substituting $P$ and $Q$ into the first equation of (\ref{ege1}) implies that the flag curvature $K=0$ if and only if $g=-\frac{2}{r+4r^2}$ and $c=\pm 2$; otherwise its flag curvature $K\neq 0$.

When $g=-\frac{2}{r+4r^2}$ and $c=\pm 2$, we have 
\begin{equation*}
\left\{\begin{array}{l}
P=-\frac{2s}{r+4r^2}\pm\frac{2\sqrt{r(r+4r^2-4s^2)}}{r(1+4r)}\\
\\
Q=-\frac{r^2-s^2}{r^3}.
\end{array}\right.
\end{equation*}
According to above step 4, we must solve the Finsler metric function $F$ from $P$ and $Q$ by a technique we  have used in the previous section 4. If let 
\[U:=\frac{s\phi+(r^2-s^2)\phi_s}{\phi},\qquad W:=\frac{s\phi_r+r\phi_s}{\phi},\]
then it is easy to see
\begin{equation}\label{ege2}
\phi_s=\frac{U-s}{r^2-s^2}\phi,\qquad\phi_r=\frac{1}{s}(W-\frac{r(U-s)}{r^2-s^2})\phi.
\end{equation}
Therefore, $P$ and $Q$ can be expressed by $U$ and $W$ via
\begin{equation}\label{ege3}
\left\{ \begin{array}{l}
 P=-QU+\frac{W}{2r}\\
 \\
Q=\frac{1}{2rs}\frac{2rU-2rs-2r^2W+s(r^2-s^2)W_s+s^2W+sUW}{U^2-sU+(r^2-s^2)U_s}.   
 \end{array} \right.
          \end{equation}
          
 Plugging the formulas of $P$ and $Q$ into the equations (\ref{ege3}), we obtain
 \[U=\frac{r\big(\pm4r^2(r+4r^2-4s^2)+s(1+2r)\sqrt{r(r+4r^2-4s^2)}\mp 2rs^2\big)}{\big(r+4r^2-4s^2\big)\big(\pm2rs+(1+2r)\sqrt{r(r+4r^2-4s^2)}\big)}.\]
 By (\ref{ege2}), one can conclude 
 \begin{equation}\label{ege4}
 (\ln\phi)_s=\frac{\pm4r(r+4r^2-2s^2)-4s(1+2r)\sqrt{r\big(r+4r^2-4s^2\big)}}{\Big(r+4r^2-4s^2\Big)\Big(\pm2rs+(1+2r)\sqrt{r\big(r+4r^2-4s^2\big)}\Big)}.
 \end{equation}
 At the same time, combining (\ref{ege3}) and (\ref{ege2}), it can be deduced that
 \begin{eqnarray}\label{ege5}
 (\ln\phi)_r&=&\frac{2\sqrt{r(r+4r^2-4s^2)}(s^2+8r^2s^2+14rs^2+8r^4-2r^3-r^2)}{r(1+4r)(r+4r^2-4s^2)\big(\pm2rs+(1+2r)\sqrt{r(r+4r^2-4s^2)}\big)}\nonumber\\
 &&-\frac{\pm4(5r^3s+20r^4s-12r^2s^3+rs^3)}{r(1+4r)(r+4r^2-4s^2)\big(\pm2rs+(1+2r)\sqrt{r(r+4r^2-4s^2)}\big)}.
\end{eqnarray}
Therefore, by using Maple, it can be verified that $\phi$ satisfies
\[(\ln\phi)_{rs}=(\ln\phi)_{sr}.\]
This means there exists $\phi(r,s)$ so that it is a solution of the equations (\ref{ege4}) and (\ref{ege5}). Integrating (\ref{ege4}) will obtian
\[\phi(r,s)=\exp\Big(\int_{0}^s\frac{\pm4r(r+4r^2-2s^2)-4s(1+2r)\sqrt{r\big(r+4r^2-4s^2\big)}}{\Big(r+4r^2-4s^2\Big)\Big(\pm2rs+(1+2r)\sqrt{r\big(r+4r^2-4s^2\big)}\Big)}ds\Big)c_0(r).\]
In order to find out $c_0(r)$, firstly notice that $\phi(r,0)=c_0(r)$. Thus $$\phi(r,0)_r=c_0'(r).$$
According to (\ref{ege5}), we know the following equation of $c_0(r)$ holds
\[\frac{c_0'}{c_0}=\frac{2(2r-1)}{(2r+1)(4r+1)}.\]
Integrating above equality will yield
\[c_0(r)=\frac{(2r+1)^2}{(4r+1)^{\frac{3}{2}}}.\]
Thus the example \ref{eg1} has a vanishing flag curvature.
\end{proof}

\begin{example}\label{eg2}
\emph{The Finsler metric 
\[F^2:=u^2\Big(\frac{1}{4r+1}\pm\frac{4\sqrt{r(r+4r^2-4s^2)}s}{r(2r+1)(4r+1)^2}-\frac{4(4r^2+3r+1)}{r(2r+1)^2(4r+1)^2}s^2\Big)\]
 defined on $\Omega=\mathbb{R}^n\setminus\{0\}$ has a negative flag curvature $K=-1$. Here $r:=|x|$, $u:=|y|$, $s:=\frac{\langle x,y\rangle}{|y|}$.}
\end{example}

\begin{proof}
As in the proof of example \ref{eg1}, if let \begin{equation*}
\left\{\begin{array}{l}
P=g(r)s+\frac{c\sqrt{r(r+4r^2-4s^2)}}{r(1+4r)}\\
\\
Q=-\frac{r^2-s^2}{r^3}.
\end{array}\right.
\end{equation*}
One can easily see that $P$ and $Q$ satisfy  both the second equation and the third equation of (\ref{ege1}).  Thus step 1 and step 2 is done.

Now suppose the flag curvature of the Finsler metric is not zero, that means that $g\neq-\frac{2}{r+4r^2}$ and $c\neq\pm 2$. In order to determine $g(r)$ and $c$, let us firstly skip step 3 and directly go to step 4. 

Now we will use a similar technique in example \ref{eg1} to find the specific $g(r)$ and $c$. Let 
\[U:=\frac{s\phi+(r^2-s^2)\phi_s}{\phi},\qquad W:=\frac{s\phi_r+r\phi_s}{\phi},\]
thus
\begin{equation*}
\left\{ \begin{array}{l}
 P=-QU+\frac{W}{2r}\\
 \\
Q=\frac{1}{2rs}\frac{2rU-2rs-2r^2W+s(r^2-s^2)W_s+s^2W+sUW}{U^2-sU+(r^2-s^2)U_s}.   
 \end{array} \right.
          \end{equation*}
 Plugging $P$ and $Q$ into above equality, one can obtain
 $$U=\frac{r[A\big(4rs+s+(4r^3s+r^2s)g(r)\big)+c(8r^4+2r^3-8r^2s^2-rs^2)]}{A\big(8r^3+6r^2+r+(4r^2s^2+rs^2)g(r)-8rs^2-2s^2\big)+c(r^2s+4r^3s-4rs^3)}$$
 where $A:=\sqrt{r(r+4r^2-4s^2)}$. Therefore, the definition of $U$ implies
 \begin{eqnarray*}
 &&(\ln\phi)_s=\frac{U-s}{r^2-s^2}\\
 &&=\frac{A\big((4r^2s+rs)g(r)-8rs-2s\big)+c(8r^3+2r^2-4rs^2)}{A\big(8r^3+6r^2+r+(4r^2s^2+rs^2)g(r)-8rs^2-2s^2\big)+c(r^2s+4r^3s-4rs^3)}.
 \end{eqnarray*}
At the same time, one can also solve $W$ and obtain $(\ln\phi)_r$ from the definition of $W$:
 \begin{eqnarray*}
 &&(\ln\phi)_r=\frac{W-r(\ln\phi)_s}{s}\\
 &&=\frac{A\big(16rs^2+2c^2r^3+8c^2r^4+32r^2s^2+2s^2-8c^2r^2s^2\big)}{r(1+4r)[A\big(8r^3+6r^2+r+(4r^2s^2+rs^2)g(r)-8rs^2-2s^2\big)+c(r^2s+4r^3s-4rs^3)]}\\
 &&+\frac{A(32r^6+32r^5+10r^4-32r^4s^2-16r^3s^2+r^3-2r^2s^2)g(r)}{r(1+4r)[A\big(8r^3+6r^2+r+(4r^2s^2+rs^2)g(r)-8rs^2-2s^2\big)+c(r^2s+4r^3s-4rs^3)]}\\
 &&+\frac{A(16r^4s^2+32r^5s^2+2r^3s^2)g(r)^2}{r(1+4r)[A\big(8r^3+6r^2+r+(4r^2s^2+rs^2)g(r)-8rs^2-2s^2\big)+c(r^2s+4r^3s-4rs^3)]}\\
 &&+\frac{c\big(-8r^4s-2r^3-8r^2s^3s-2rs^3+(64r^6s+32r^5s+4r^4s-64r^4s^3-16r^3s^3)g(r)\big)}{r(1+4r)[A\big(8r^3+6r^2+r+(4r^2s^2+rs^2)g(r)-8rs^2-2s^2\big)+c(r^2s+4r^3s-4rs^3)]}.
 \end{eqnarray*}
Since $(\ln\phi)_{sr}=(\ln\phi)_{rs}$, we can compute it with the help of Maple and conclude that it holds if and only if 
 $$g(r)=\frac{2(-10r+c^2r-3)}{3r(2r+1)(4r+1)},\quad c=\pm 1\ \text{or}\ \pm 2.$$
However, for $c\neq \pm2$, $c$ has to be $\pm 1$ and $g(r)=-\frac{2(3r+1)}{r(2r+1)(4r+1)}$.

Now $P$ and $Q$ are given by
\begin{equation*}
\left\{\begin{array}{l}
P=\frac{-2(3r+1)s}{r(2r+1)(4r+1)}\pm\frac{\sqrt{r(r+4r^2-4s^2)}}{r(4r+1)}\\
\\
Q=-\frac{r^2-s^2}{r^3}
\end{array}\right.
\end{equation*}
and we can substitute $P$ and $Q$ into the first equation of (\ref{ege1}) to execute  step 3. Finally the metric $F$ can be obtained:
\[F^2=u^2\Big(\frac{1}{4r+1}\pm\frac{4\sqrt{r(r+4r^2-4s^2)}s}{r(2r+1)(4r+1)^2}-\frac{4(4r^2+3r+1)}{r(2r+1)^2(4r+1)^2}s^2\Big)\]
and the flag curvature $K$ is -1.
\end{proof}

\begin{remark} Actually, one can choose $c_1(r)=\frac{a}{r}$ and do the above process to obtain a family of Finsler metrics with constant flag curvature depending on $a$.  Since the metrics are similar with the examples, we do not go to the details and give an explicit expression here. 
\end{remark}

\begin{example}\label{eg3}
\emph{The Finsler metric 
\[F:=u\exp\Big(4r^2\int_{0}^s\frac{4\big(\pm(4r^2-2s^2+1)-(4r^2s-4s^3+s)\sqrt{\frac{1+4r^2}{1+4r^2-4s^2}}\big)ds}{(16r^4+8r^2-32r^2s^2+16s^4-8s^2+1)\sqrt{\frac{1+4r^2}{1+4r^2-4s^2}}\pm(8r^2s-8s^3+2s)}\Big)\]
 defined on $\Omega=\mathbb{R}^n$ has a vanishing flag curvature. Here $r:=|x|$, $u:=|y|$, $s:=\frac{\langle x,y\rangle}{|y|}$.}
\end{example}

\begin{proof}
We know that $Q=c_1(r)+\frac{c_1'+2rc_1^2}{r-2r^3c_1}s^2$ is a solution of the third equation of (\ref{ege1}). If assume  $c_1(r)=-2$, then 
 \begin{equation*}
\left\{\begin{array}{l}
P=h(r)s+c\sqrt{\frac{1+4r^2-4s^2}{1+4r^2}}\\
\\
Q=-2+\frac{8s^2}{1+4r^2}
\end{array}\right.
\end{equation*}
satisfy the second equation and the third equation of (\ref{ege1}).
Plugging $P$ and $Q$ into the first equation of (\ref{ege1}), one can see that the flag curvature $K = 0$ if and only if $h=0$ and $c=\pm2$.
Therefore \begin{equation*}
\left\{\begin{array}{l}
P=\pm2\sqrt{\frac{1+4r^2-4s^2}{1+4r^2}}\\
\\
Q=-2+\frac{8s^2}{1+4r^2}.
\end{array}\right.
\end{equation*}

Now let us solve the metric function $F$. By a similar argument in example (\ref{eg1}), one can calculate that
\begin{equation*}
\left\{\begin{array}{l}
(\ln\phi)_s=\frac{4\big(\pm(4r^2-2s^2+1)-(4r^2s-4s^3+s)\sqrt{\frac{1+4r^2}{1+4r^2-4s^2}}\big)}{(16r^4+8r^2-32r^2s^2+16s^4-8s^2+1)\sqrt{\frac{1+4r^2}{1+4r^2-4s^2}}\pm(8r^2s-8s^3+2s)}\\
\\
(\ln\phi)_r=\frac{8r[(16r^4+8r^2-24r^2s^2+8rs^4-6s^2+1)\sqrt{\frac{1+4r^2}{1+4r^2-4s^2}}\pm(4s^3-8r^2s-2s)]}{(16r^4+8r^2-32r^2s^2+16s^4-8s^2+1)\sqrt{\frac{1+4r^2}{1+4r^2-4s^2}}\pm(8r^2s-8s^3+2s)}.
\end{array}\right.
\end{equation*}
It is easy to check $(\ln\phi)_{sr}=(\ln\phi)_{rs}$ does hold by Maple. This means there exist $\phi$ satisfies above equations and one can solve that
\[\phi=\exp\Big(\int_{0}^s\frac{4\big(\pm(4r^2-2s^2+1)-(4r^2s-4s^3+s)\sqrt{\frac{1+4r^2}{1+4r^2-4s^2}}\big)ds}{(16r^4+8r^2-32r^2s^2+16s^4-8s^2+1)\sqrt{\frac{1+4r^2}{1+4r^2-4s^2}}\pm(8r^2s-8s^3+2s)}\Big)c_0(r)\]
where $c_0=\exp(4r^2)$.
\end{proof}

\begin{example}\label{eg4}
\emph{The Finsler metric 
\[F^2:=\frac{u^2}{(1+4r^2)^2}\big(16r^4+8r^2-16r^2s^2+1\pm4s\sqrt{(1+4r^2)(1+4r^2-4s^2)}\big)\]
 defined on $\Omega=\mathbb{R}^n$ has a negative flag curvature $K=-1$. Here $r:=|x|$, $u:=|y|$, $s:=\frac{\langle x,y\rangle}{|y|}$.}
\end{example}
\begin{proof} In the proof of example (\ref{eg3}), let  
 \begin{equation*}
\left\{\begin{array}{l}
P=h(r)s+c\sqrt{\frac{1+4r^2-4s^2}{1+4r^2}}\\
\\
Q=-2+\frac{8s^2}{1+4r^2}
\end{array}\right.
\end{equation*}
and assume that $h\neq 0$ and $c\neq\pm 2$. Under these conditions, the metric does not have a vanishing flag curvature. From the expression of $P$ and $Q$, one can solve $(\ln\phi)_{s}$ and $(\ln\phi)_{r}$ in terms of $h$ and $c$. The equation $(\ln\phi)_{sr}=(\ln\phi)_{rs}$ implies 
$$h(r)=\frac{2(c^2-4)}{3(1+4r)},\quad c=\pm 1\ \text{or}\ \pm 2.$$
By our assumption, $h=-\frac{2}{1+4r}$ and $c=\pm 1$. Plugging 
\begin{equation*}
\left\{\begin{array}{l}
P=-\frac{2s}{1+4r}\pm\sqrt{\frac{1+4r^2-4s^2}{1+4r^2}}\\
\\
Q=-2+\frac{8s^2}{1+4r^2}
\end{array}\right.
\end{equation*}
into the first equation of (\ref{ege1}) will have
$$\phi^2=\frac{1}{(1+4r^2)^2}\big(16r^4+8r^2-16r^2s^2+1\pm4s\sqrt{(1+4r^2)(1+4r^2-4s^2)}\big)$$
and the flag curvature $K=-1$.
\end{proof}

\begin{remark}
(1) Here, one can also vary $c_1=a$ to produce another similar class of Finsler metrics with constant flag curvature depending on $a$.
It is very possible to find some new examples of constant flag curvature $K=1$ via this method.

 (2) However, it seems that it is quite hard to get an explicit classification theorem on this class of Finsler metrics with constant flag curvature since even for some specific $c_1(r)$, one can not alway get an explicit formula of $P$ from the second equation of (\ref{ege1}).
\end{remark}

{\small SCHOOL OF MATHEMATICS, PEKING UNIVERSITY  }

{\small BEIJING 100871, CHINA} 

{\small E-mail address: moxh@pku.edu.cn}

\bigskip

{\small DEPARTMENT OF MATHEMATICS, EAST CHINA NORMAL UNIVERSITY }

{\small SHANGHAI 200241, CHINA}

{\small E-mail address: lfzhou@math.ecnu.edu.cn}

\end{document}